\documentclass[a4paper,11pt]{article}

\usepackage[utf8]{inputenc}
\usepackage{amsmath,amssymb,amsthm,amsfonts}
\usepackage{hyperref}
\usepackage{graphicx}
\usepackage{caption}
\usepackage{subcaption}
\usepackage{enumerate}

\def\R{\mathbb R}

\def\p{\partial}

\def\P{\mathbb P}

\def\E{\mathbb E}

\def\bigO{\mathcal O}

\renewcommand{\vec}[1]{\boldsymbol{#1}}

\newtheorem{theorem}{Theorem}[section]
\newtheorem*{theorem*}{Theorem}
\newtheorem{proposition}[theorem]{Proposition}
\newtheorem{lemma}[theorem]{Lemma}

\newtheorem*{corollary*}{Corollary}

\newtheorem*{conjecture*}{Conjecture}

\theoremstyle{remark}

\newtheorem{remark}{Remark}

\usepackage{color}
\definecolor{light}{gray}{0.8}

\author{}

\begin{document}

\title{Markov-modulated M/G/1 type queue in heavy traffic and its application to
time-sharing disciplines}

\author{H. Thorsdottir$\,^a$ and I. M. Verloop$\,^{b,c}$}

\date{\today}

\maketitle

\begin{abstract}
\noindent 
This paper deals with a single-server queue with modulated arrivals, service requirements and service capacity. In our first result, we derive the mean of the total workload assuming generally distributed service requirements and any service discipline which does not depend on the modulating environment. We then show that the workload is exponentially distributed under heavy-traffic scaling. In our second result, we focus on the discriminatory processor sharing (DPS) discipline. Assuming exponential, class-dependent service requirements, we show that the joint distribution of the queue lengths of different customer classes under DPS undergoes a state-space collapse when subject to heavy-traffic scaling. That is, the limiting distribution of the queue length vector is shown to be  exponential, times a deterministic vector. The distribution of the scaled workload, as derived for general service disciplines, is a key quantity in the proof of the state-space collapse.
\vspace{3mm}


\vspace{2mm}


\vspace{3mm}

\noindent $^a$ CWI, Science Park 123, 1098XG Amsterdam, the Netherlands.\\
$^b$ CNRS, IRIT, 2 rue C. Camichel, 31071 Toulouse, France.\\
$^c$ Universit\`e de Toulouse, INP, 31071 Toulouse, France.\\
e-mail: halldora@cwi.nl, verloop@irit.fr

\vspace{3mm}

\noindent{{\bf Acknowledgements:} The authors would like to thank Urtzi Ayesta, Joke Blom and Michel Mandjes for helpful discussions. This research was partially supported by the SMI program of INP Toulouse.}

\end{abstract}

\section{Introduction}
Markov-modulation is a way to formalize the embedding of queues in a random environment. The parameters of the queue in question, typically arrival rates, service requirements or both, are governed by an external Markov chain, thereby creating an extra layer of randomness around the stochastic queueing process.
For classical results on Markov-modulated single-server queues with the first-come-first-serve (FCFS) discipline see e.g. \cite{Neuts81, Purdue, Regter}. Recent work on systems in a Markov-modulated environment can be found in for example~\cite{Fiems13, Kim15, Zorine14}.

In this paper we will analyse a modulated queue under a heavy-traffic scaling, that is, evaluate the system at its critical load. 
It is a well-known result from the literature of single-server queues, that under fairly general conditions~\cite{Kingman}, the steady-state distributions of the appropriately scaled queue length and workload become exponential when the critical load is approached. 
This property has been seen to carry over to certain systems where arrivals and service times are modulated by an external Markov process, see \cite{Asm87, Dimi11, FalinFalin}. In fact, \cite{Asm87} establishes an even stronger result: convergence of the queue length \emph{process} to a reflected Brownian motion.
\emph{Multi-class single-server queues} 
under a heavy-traffic scaling have been studied in e.g.\ \cite{Katsuda} for FCFS with feedback routing, \cite{Ayesta11} for the discriminatory random order of service  discipline, and \cite{Grishechkin92,  Verloop11} for the discriminatory processor sharing (DPS) policy. In particular, \cite{Ayesta11, Verloop11}  show that the steady-state queue length vector undergoes a so-called \emph{state-space collapse} and  converges to an exponentially distributed variable, times a deterministic vector. 
The cited multi-class results under heavy-traffic scaling are all for non-modulated systems. 
In light of this, we will in this paper put special emphasis on a modulated multi-class single-server queue, and the limiting steady-state queue length distribution is derived.

While there is little ambiguity in how arrival rates are modulated,
there are in the literature typically two ways in which to modulate the service rates. One can (i) let the departure rate be {\it continuously} modulated throughout a customer's service, the other approach is to (ii) let a customer's service requirement distribution be based on the state of the environment when it arrives and remain the same until the customer departs. We note that by adapting the number of different customer classes, the fixed service requirements of case (ii) can be seen as a special case of the continuously modulated requirements (i); we further elaborate on this later in the paper in Remark \ref{rem:contVSfixed} in Section \ref{sec:model}.

 The way the load or traffic intensity for modulated queues is defined goes hand in hand with the way the service rates 
 are affected by the environment. 
In case (i), the load is typically the average of arrival rates (where the averaging is with respect to the equilibrium distribution of the environment), say $\lambda_\infty$, divided by the average of service rates, say $\mu_\infty$ (see e.g. \cite{Neuts81}). 
In case (ii), the load is taken as the average over the arrival rate times the mean service requirement, say, $\lambda_d / \mu_d$ per state $d$ of the environment (see e.g. \cite{Dimi11,Regter}).
The two load definitions represent different scenarios. In particular, when load (ii) is equal to 1 (the critical load) it means that in at least one state of the environment, the total load over all classes must exceed 1, i.e. for at least one state we must have overload. This is true only for special cases of definition (i).

In this paper, special focus will be given to a multi-class single-server queue under the DPS discipline.
The DPS discipline was first introduced by Kleinrock in \cite{Klein67} as an extension of the well-known egalitarian processor sharing (PS) discipline and has turned out to be very suitable to model the simultaneous parsing of diverse tasks, such as processing network data.
Under this service discipline, the service capacity is divided between all present customers in proportion to their prescribed weights.
Due to the challenging nature of DPS systems, most available results are in terms of limit theorems and moments. 
Fayolle et al. \cite{Fayolle80} established the mean sojourn time conditioned on the service requirement. That analysis also yielded the mean queue lengths of the different classes, which were shown to depend on the entire service requirement distributions of all classes. 
The DPS model has finite mean queue lengths irrespective of any higher-order characteristics of the service distribution, see Avrachenkov et al. \cite{AVRA}. This is an extension of a result for the Processor Sharing (PS) system, which holds while the queue is stable.
DPS under a heavy-traffic regime was analysed in \cite{Grishechkin92} assuming finite second moments of the service requirement distributions. Assuming exponential service requirement distributions, a direct approach to establish a heavy-traffic limit for the joint queue length distribution was described in \cite{Rege96} and extended to phase-type distributions in \cite{Verloop11}. Combining light and heavy-traffic limits, in \cite{Izagirre15} an interpolation approximation is derived for the steady-state distribution of the queue length and waiting time of DPS. The performance of DPS in overload is considered in \cite{Altman04}. Asymptotics of the sojourn time have received attention in \cite{Borst05, Borst06}. Game-theoretic aspects of DPS have been studied in \cite{Wu12, Hassin03}. A thorough overview of DPS results can be found in \cite{Altman}. 

We are not aware of work analysing a DPS system under modulation. We refer to \cite{Nunez00} where the Processor Sharing discipline (DPS discipline with unit weights) was analysed in a Markovian random environment. Multi-class queues in a random environment have been studied for different models in
\cite{Budh14, Takine05}. In \cite{Takine05}, 
a modulated system 
is studied where arrivals can only occur at transition epochs of the modulating process but service requirements are class-dependent and generally distributed. 
Using a time-changing argument, the waiting time distribution is derived under the FCFS discipline. In \cite{Budh14}, the authors derive a Brownian control problem to establish a form of the $c\mu$ scheduling rule in heavy traffic under continuously modulated service requirements. By using a particular scaling, the time-scale separation of the external environment and the queue length process is exploited. Similar scaling of a modulated queue can also be seen in results on the Markov-modulated infinite server queue in e.g.\ \cite{KELL}.

The system we analyse in this paper is a single-server queue where the arrival rates, service requirements and service capacity are modulated.
We focus on the setting where a customer's service requirement distribution is based on the state of the environment when it arrives and does not change throughout its service. The service capacity is however continuously modulated. This assumption is in line with the literature for various types of modulated queues, see \cite{boxma2001,Dorsman, Mahab05, Takine05}. In Remark \ref{rem:split} in Section \ref{sec:DPSproof}, we discuss how part of our results can be extended to a more general model with continuously modulated service requirements.  We derive the distribution of the workload under a heavy-traffic scaling for generally distributed service requirements and any service discipline which does not depend on the environment. 
We then turn our attention to the DPS discipline in a multi-class queue, which is a particular case of the general modulated system as described above. 
The weights of the DPS system determining the service proportion, depend on a customer's class and do not change with the environment, which means that the workload result remains valid. 
An important finding in the present paper is that the queue length vector under DPS becomes independent of the modulating process in the heavy-traffic limit, which is consistent with the modulated M/G/1 queue in e.g.
\cite{Asm87,Dimi11}. This, together with the obtained result on the workload, allows us to derive the distribution of the queue length vector under DPS and to show that it undergoes a state-space collapse. 

The remainder of the paper is organized as follows. In Section \ref{sec:model} we describe the model. In Section \ref{sec:workload} we study the workload of a single-server queue with modulated arrivals, service capacity and service requirement distribution, establishing its distribution in heavy traffic. From Section \ref{sec:queuelength} onwards, the focus is on DPS. In Section \ref{sec:queuelength} we derive some basic properties of the queue length distribution, obtain a rate conservation law and derive an equation for the moments of the queue lengths weighted with the modulated service capacity. Section~\ref{sec:HT} is devoted to the heavy-traffic scaling; there we show that the distribution of the environment becomes independent of that of the queue length vector, in addition to deriving two technical lemmas. The exponential limiting distribution of the joint queue length in heavy traffic follows in Section \ref{sec:DPSproof}. The result is shown in two steps in the subsections \ref{sec:SSC} about the state-space collapse and \ref{sec:commonfactor} about the exact limiting distribution, where we rely on the workload result of Section \ref{sec:workload}. We conclude with a summary and some open questions in Section \ref{sec:conclusion}.

\section{Model}\label{sec:model}
We analyse a single-server queue modulated by an independent external environment, which is formalized by an irreducible continuous time Markov chain on a finite state-space $\{1,\ldots,D\}$. The modulating process is denoted with $Z$ and is governed by an infinitesimal generator matrix $Q = \left(q_{d\ell}\right)_{d,\ell=1}^D$ with an invariant distribution $\vec{\pi}=(\pi_1,\ldots,\pi_D)$. 
In what follows, vectors are generally denoted in bold.
New customers arrive according to a Poisson distribution with rate $\lambda_d$ when the environment is in state $d$. A customer arriving in state $d$ has a service requirement distribution given by a function $H_d(\cdot)$ with Laplace-Stieltjes transform (LST) $h_d(\cdot)$. The service requirement does not change further with the environment. The first and second moment are given by $h_{d1}$ and $h_{d2}$, respectively.
In addition we let the service capacity be scaled by a factor $c_d$ during the environment's stay in state $d$, this can thus change during the service of the customers.
The traffic intensity will be measured as 
\begin{equation}\label{eq:rho_infty}
 \rho_\infty = \sum_d \pi_d \lambda_d h_{d1}/c_\infty, 
\end{equation}
with $c_\infty := \sum_d \pi_d c_d$ being the service capacity averaged over the environment. The workload is defined as the time it takes to empty the system at an arbitrary moment in time given the observed environment and is denoted by~$W$.
In Section~\ref{sec:workload} we study the workload and the environment as a two-dimensional process $(W,Z)$, under any service discipline that is independent of the environment.

The first main result of this paper concerns the distribution of the workload
when the traffic intensity approaches its critical point. The system is said to be in heavy traffic when $\rho_\infty$ approaches 1.
Let $N>0$ and define the following parametrization
\begin{equation}\label{eq:HTlambda}
 \lambda^{(N)}_d := \frac{\lambda_d}{\rho_\infty}(1-1/N)
 \to \frac{\lambda_d}{\rho_\infty} =: \hat{\lambda}_d, \quad\text{as } N\to\infty,
\end{equation}
where $\rho_\infty$ is based on the unscaled parameters. 
Prelimit quantities will be denoted with a superscript $(N)$; the prelimit traffic intensity is thus $\rho_\infty^{(N)}$ and is equal to $1-1/N$. 
Limiting quantities will have a \^{}; in heavy traffic the traffic intensity is denoted $\hat{\rho}_\infty$ and is equal to 1.

In the remainder of the paper, starting in Section \ref{sec:queuelength}, we analyse a single-server queue with $K$ customer classes under the discriminatory processor sharing policy; the queue is again embedded in a random environment. 
Let $\alpha_{k,d}$ be the probability that a customer, that arrives while the environment is in state $d$, is of class $k$; note that $\sum_k \alpha_{k,d}=1$ for a given $d$. 
The Poisson arrival rate of a class-$k$ customer is denoted with $\lambda_{k,d} := \alpha_{k,d}\lambda_d$ and for each class $k$ it is assumed that $\lambda_{k,d}>0$ for at least one state $d$. 
In the multi-class setting we assume that a class-$k$ customer has an exponentially distributed service requirement with mean $1/\mu_k$. 
We believe that the results obtained in this paper can be extended to phase-type distributed service requirements, the latter being dense in the space of all distributions on $[0,\infty)$. For the non-modulated DPS queue, the phase-type analysis was performed in~\cite{Verloop11} using similar proof techniques. For ease of exposition, however, we focus here on the exponential case.

We no longer let the service requirement of a particular customer be environment-dependent (although the distribution of an arbitrary customer is, as explained in Section \ref{sec:commonfactor}).
One can however retrieve the environment-dependent service requirements by introducing additional classes for each environment, see Remark \ref{rem:contVSfixed}.
By referring to a class-$k$ customer's service rate while in state $d$ as $\mu_{k,d} := \mu_k c_d$, we take the modulated service capacity into account.
Most of the results for the queue length can in fact be shown without assuming this product form, representing a system where the service requirements are continuously modulated, see Remark \ref{rem:split}, Section \ref{sec:SSC}, for further details.

We denote the average arrival rate of class-$k$ customers by $\lambda_{k,\infty} := \sum_d \lambda_{k,d}\pi_d$, similarly we denote the average service rate for class $k$ by $\mu_{k,\infty} := \sum_d \mu_{k,d}\pi_d = \mu_k c_\infty$ and $\rho_{k,\infty} := \lambda_{k,\infty} / \mu_{k,\infty}$.
The aggregate traffic intensity for the multi-class model is defined as
\[
 \rho_\infty := \sum_{k=1}^K \rho_{k,\infty},
\]
which is consistent with the definition in Eqn. \eqref{eq:rho_infty}. 

Let the state of the multi-class system be described by the vector of random variables $(M_1,\ldots, M_K,Z) =: (\vec{M},Z)$, where $M_k$ is the number of class-$k$ customers, for $k=1,\ldots,K$. As before, $Z$ represents the state of the background process.
In a DPS system, the random fraction of service given to a class-$k$ customer is
\[
 \frac{g_k}{\sum_j g_j M_j},
\]
where 
$g_k$ are weight parameters associated with each class $k$.\\

When in heavy traffic, we denote $\hat{\rho}_{k,\infty} := \hat{\lambda}_{k,\infty} / \mu_{k,\infty}$ and thus have for the multi-class system 
\[
 \sum_k \hat{\rho}_{k,\infty} = \sum_k \frac{\hat{\lambda}_{k,\infty}}{\mu_{k,\infty}} = \frac{1}{\rho_\infty}\sum_k\frac{\lambda_{k,\infty}}{\mu_{k,\infty}} = \frac{1}{\rho_\infty}\sum_k \rho_{k,\infty} = 1.
\]

\begin{remark}[Modulated service requirements rewritten to classes]\label{rem:contVSfixed}
 Any multi-class system where the service requirement distribution is determined by the modulating process at a customer's arrival,
 can be written as a multi-class model with non-modulated, only class-dependent, service rates $\mu_k$, as illustrated below:
While in state $d$ of the environment, class-$k$ customers arrive with rate $\lambda_{k,d}$ 
and 
have exponential service requirement with mean $1/\mu_{k,d}$
and weight $g_k$. 
Such customers we refer to as being of class $(k,d)$ and count with $M_{k,d}$, hence we need to keep track of $K\cdot D$ different customer ``classes". Arrivals to class $(k,d)$ are inactive while not in state $d$. 
    \begin{table}[h]
    \centering
 \begin{tabular}{ | l || c | c | c | c |} \hline 
   classes    & \multicolumn{2}{ c| }{arrival rates} & serv. rate & weight \\ \hline
       & $d=1$ & $d=2$ & \multicolumn{2}{ c| }{$d\in\{1,2\}$} \\ \hline
 (1,1) & $\lambda_{1,1}$ & 0 & $\mu_{1,1}$ & $g_1$ \\ \hline 
 (1,2) & 0 & $\lambda_{1,2}$ & $\mu_{1,2}$ & $g_1$ \\ \hline
 (2,1) & $\lambda_{2,1}$ & 0 & $\mu_{2,1}$ & $g_2$ \\ \hline 
 (2,2) & 0 & $\lambda_{2,2}$ & $\mu_{2,2}$ & $g_2$ \\ \hline 
\end{tabular}
 
 \caption{A multi-class system where the service requirement distribution is fixed upon arrival can be translated into one with non-modulated service requirements.}
 \label{tab:remark1}
 \end{table}
 
 From Table \ref{tab:remark1} one can easily see how a $K=D=2$ system can be written into one
  with $K=4$ and $D=2$. 
 The arrival rates are still modulated, but in an on-off way.
The service rates $\mu_{k,d}$ are now non-modulated.
 \end{remark}

 \section{Workload}\label{sec:workload}
In this section we consider the workload in a modulated queue and extend the results for an M/G/1 type queue from Falin and Falin \cite{FalinFalin} and Dimitrov \cite{Dimi11} to include modulated service capacity.
We derive the mean of the workload and then its distribution in the heavy-traffic regime.

Let $p_{0,d} = P(W = 0,Z=d)$ 
and let $\vec{a} = (a_1,\ldots,a_D)^T$ be a vector solving
\begin{equation}\label{eq:Qa}
 [Q\cdot \vec{a}]_d 
 =   c_d - \lambda_d h_{d1} -c_\infty(1-\rho_\infty  ),
\end{equation}
for $d=1,\ldots,D$.
Note that
such a solution always exists since the right hand side vector of Eqn. \eqref{eq:Qa} is orthogonal to $\vec{\pi}$. We obtain the following result for the mean workload.

\begin{proposition}\label{prop:EW}
For any service requirement distribution $H_d(\cdot)$ and any service discipline that does not depend on the state of the environment, the mean of the workload satisfies
\begin{equation}\label{eq:Wdweighted}
 \E W = \frac{\sum_d \left[\pi_d \lambda_d h_{d2}/2 + a_d\pi_d (\lambda_d h_{d1} - c_d) + p_{0,d}c_d a_d\right]}{c_\infty(1 - \rho_\infty)},
\end{equation}
where $\vec{a}$ is a solution of Eqn. \eqref{eq:Qa}.
Furthermore, $1 - \rho_\infty = \sum_d p_{0,d}c_d/c_\infty.$
 \end{proposition}

 \begin{remark}
Although the solution vector $\vec{a}$ is not unique, the term 
\[
 \sum_{d=1}^D a_d[\pi_d (\lambda_d h_{d1} - c_d) + p_{0,d} c_d],
\]
as appearing in Eqn. \eqref{eq:Wdweighted}, is. This is due to the following argument: Suppose $\vec{a}$ and $\vec{a}^*$ are two solutions to Eqn. \eqref{eq:Qa}.
Then $0 =  Q\vec{a} - Q\vec{a^*} = Q(\vec{a} - \vec{a}^*)$, so $(\vec{a} - \vec{a}^*)$ is in the nullspace of $Q$. But $Q$ is a generator matrix so it can easily be seen that $Q\vec{r} = 0$ for any vector $r\cdot \vec{1}^T = (r,r,\ldots,r)^T$, $r\in\R$. Also, since the environment is an irreducible Markov chain, the nullspace of $Q$ has dimension 1, and therefore, $(\vec{a} - \vec{a}^*) = r_a\cdot \vec{1}^T$, for some $r_a\in\R$. Thus,
\[
 \sum_{d=1}^D (a_d - a^*_d)[\pi_d(\lambda_d h_{d1} - c_d) + p_{0,d}c_d]
 = r_a c_\infty (\rho_\infty - 1) 
 + r_a c_\infty (1 - \rho_\infty) = 0,
\]
where the first term follows by definition of $\rho_\infty$ and $c_\infty$ and the second term comes from Eqn. \eqref{eq:pd0cdlimit} in the proof below.
\end{remark}

 \begin{proof}[Proof of Proposition \ref{prop:EW}]
Define $F_d(x,t) = P(W(t) < x, Z(t) = d)$, for some time $t>0$. In an infinitesimal time ${\rm d}t$, a new arrival requiring service $x$ changes the workload with probability $\lambda_d H_d(x){\rm d}t$. The service capacity is scaled by $c_d$ when the environment is in state $d$, meaning that in ${\rm d}t$ time, the workload is reduced by $c_d{\rm d}t$,
yielding by a classic birth-and-death argument for the M/G/1 queue,
\begin{align*}
F_d(x,t+{\rm d}t) &= (1-\lambda_d {\rm d}t + q_{dd} {\rm d}t)F_d(x + c_d{\rm d}t,t) \\
&+ \sum_{\ell\neq d}q_{\ell d}F_\ell(x + c_\ell{\rm d}t,t){\rm d}t + \lambda_d {\rm d}t\int_0^x F_d(x+c_d{\rm d}t -y,t) {\rm d} H_d(y).
\end{align*}
We let $t\to\infty$ to go to steady-state and since
\[
\frac{F_d(x + c_d{\rm d}t) - F_d(x)}{{\rm d}t} = c_d \frac{F_d(x + c_d{\rm d}t) - F_d(x)}{c_d{\rm d}t} 
\underset{{\rm d}t \downarrow 0}{\longrightarrow}
c_d F_d'(x),
\]
we obtain 
\[
 c_d F_d'(x) = (\lambda_d - q_{dd}) F_d(x) - \sum_{\ell\neq d}q_{\ell d}F_\ell(x) - \lambda_d \int_0^x F_d(x -y) {\rm d} H_d(y).
\]
Denote the LST of $F_d(\cdot)$ by $\varphi_d(s) = \E[\mathrm{e}^{-sW(t)}, Z(t)=d] = p_{0,d} + \int_{0+}^\infty \mathrm{e}^{-sx} {\rm d}F_d(x)$. The corresponding transform equation becomes
\begin{equation}\label{eq:phiLST}
 \sum_{\ell=1}^D q_{\ell d}\varphi_\ell(s) = [\lambda_d(1-h_d(s)) - sc_d]\varphi_d(s) + s p_{0,d}c_d.
\end{equation}
It is now convenient to sum over $d$ and divide through Eqn. \eqref{eq:phiLST} with $s$ to get zero on the left hand side, leading to 
\begin{equation}\label{eq:pd0cd}
 \sum_d p_{0,d} c_d = \sum_d \varphi_d(s)\left[c_d - \lambda_d \frac{(1-h_d(s))}{s}\right].
\end{equation}
Using that $\varphi_d(0) = \pi_d$ and by l'H\^opital
\[
 \lim_{s\downarrow 0} \frac{(1-h_d(s))}{s} = -\lim_{s\downarrow 0} h'_d(s) = h_{d1},
\]
we get by taking the limit $s\to 0$ of Eqn. \eqref{eq:pd0cd} that 
\begin{equation}\label{eq:pd0cdlimit}
 \sum_d \frac{p_{0,d}c_d}{c_\infty} = 1 - \rho_\infty.
\end{equation}

We differentiate Eqn. \eqref{eq:pd0cd} w.r.t. $s$:\\
\[
 \sum_d\varphi_d(s)\lambda_d\left[\frac{h'_d(s)}{s} + \frac{1-h_d(s)}{s^2}\right] 
 = \sum_d\varphi'_d(s)\left[c_d - \lambda_d\frac{1-h_d(s)}{s}\right],
\]

which in the limit of $s\to 0$ results in
\begin{equation}\label{eq:Wd2nd}
\sum_d \frac{\pi_d \lambda_dh_{d2}}{2} = \sum_d W_d \left[c_d - \lambda_d h_{d1}\right],
\end{equation}
with the first moment of the workload while in state $d$ being
\[
-\lim_{s\downarrow 0} \varphi'_d(s) = \E[W, Z=d] =: W_d.
\]
 
Now multiply Eqn. \eqref{eq:phiLST} with $a_d$, sum over $d$, take the derivative w.r.t. $s$ and let $s\to 0$ to obtain
\[
 -\sum_d W_d [ c_\infty(\rho_\infty - 1) + c_d - \lambda_d h_{d1}] = \sum_d \left[a_d\pi_d (\lambda_d h_{d1} - c_d) + p_{0,d}c_d a_d\right],
\]
by using Eqn. \eqref{eq:Qa}. Adding this equation to Eqn.\eqref{eq:Wd2nd} yields
\begin{equation}
 c_\infty(1 - \rho_\infty)\sum_d W_d = \sum_d \left[\frac{\pi_d \lambda_d h_{d2}}{2} + a_d\pi_d (\lambda_d h_{d1} - c_d) + p_{0,d}c_d a_d\right],
\end{equation}
which gives the desired expression for the mean of the workload, $\E W = \sum_d W_d$. 
\end{proof}

Eqn.~\eqref{eq:pd0cdlimit} makes it clear that in heavy traffic, that is when $\rho^{(N)}_\infty \to 1$,
all the probabilities $p^{(N)}_{0,d}$ go to zero. 
Also, in heavy traffic, the right hand side of Eqn. \eqref{eq:Qa} reduces to $c_d - \hat \lambda_d h_{d1}$.
Recalling the parametrization $1-\rho^{(N)}_\infty = 1/N$ in Section \ref{sec:model}, we obtain the following result.

\begin{proposition}\label{prop:EWinHT}
For any service distribution $H_d(\cdot)$ and any service discipline that does not depend on the state of the environment, the mean of the workload in heavy traffic satisfies
\begin{equation}\label{eq:EWinHT}
 \lim_{N\to\infty}\frac{1}{N} \E W^{(N)}
 = \frac{1}{c_\infty}\sum_d \pi_d\left[\hat{\lambda}_d h_{d2}/2 + a_d (\hat{\lambda}_d h_{d1} - c_d)\right],
\end{equation} 
where $\vec{a}$ is a solution of $[Q\cdot \vec{a}]_d 
 =   c_d - \hat \lambda_d h_{d1}$.
\end{proposition}
\begin{proof}
Under the heavy traffic scaling, the empty probabilities $p^{(N)}_{0,d}$ go to zero for $d=1,\ldots,D$. The result then follows immediately from Eqn. \eqref{eq:Wdweighted} and $1-\rho^{(N)}_\infty = 1/N$.
\end{proof}
This leads to the main result of this section.\\
\begin{theorem}\label{thm:wht}
In heavy traffic, the scaled workload $\frac{1}{N}W^{(N)}$, converges in distribution to  $\hat{W}$,
 where $\hat{W}$ is exponentially distributed with mean given in Eqn. \eqref{eq:EWinHT}.
 \end{theorem}
\begin{proof}
 This follows from combining Proposition \ref{prop:EWinHT} with Theorem 4 in \cite{Dimi11}. The full proof is in the Appendix.
\end{proof}

The above results yields that $\hat{W}$ is relatively compact, 
which together with the metric space being separable and complete implies that the scaled workload $\frac{1}{N}W^{(N)}$ is tight, by Prohorov's theorem \cite{Billingsley}.

\section{Queue length vector under DPS}\label{sec:queuelength}

In the remainder of the paper we focus on the multi-class model under DPS, where we assume that the service requirements of class-$k$ customers are exponentially distributed with rate $\mu_k$. In this section we establish some properties of the joint queue length distribution.
We start with the flow equations, followed by a rate conservation law and 
an equation for the moments of the queue lengths conditioned on the environment.

Equating the flow in and out of state $(\vec{M},Z)=(\vec{m},d)$ yields (noting that $-q_{dd}=\sum_{\ell\neq d}q_{d \ell}$)
\begin{align}\nonumber
 & \left(\sum_{k=1}^K (\lambda_{k,d} + \frac{g_k m_k}{\sum_i g_i m_i}\mu_{k,d}
 \cdot\vec{1}_{\{m_k>0\}}) - q_{dd}\right)p_{\vec{m},d} \\ \label{eq:floweqn}
 & =\sum_{k=1}^K(\lambda_{k,d}p_{\vec{m-e}_k,d}
 \cdot\vec{1}_{\{m_k>0\}} + \frac{g_k (m_k+1)}{\sum_i g_i m_i + g_k}\mu_{k,d}p_{\vec{m+e}_k,d}) + \sum_{\ell\neq d}q_{d\ell}p_{\vec{m},\ell},
\end{align}
where $p_{\vec{m},d}:=\P((\vec{M},Z)=(\vec{m},d))$ and $\vec{e}_k$ is the vector with 1 in the $k$-th place and zeros elsewhere.
We now define the partial probability generating function (PGF)
for when the background process is in state $d$: 
\begin{align*}
P_d(\vec{z}) &:= \E[z_1^{M_1}\cdots z_K^{M_K} \cdot \vec{1}_{\{Z=d\}}] \\
&:= \sum_{m_1=0}^\infty\cdots\sum_{m_K=0}^\infty \P(M_1=m_1,\ldots,M_K=m_K,Z=d) \cdot z_1 ^{m_1}\cdots z_K^{m_K} \\
&=  \sum_{\vec{m} \geq \vec{0}}  p_{\vec{m},d} {\vec{z^m}},
\end{align*}
where $z_1^{m_1}\cdots z_K^{m_K} =: \vec{z^m}$ and $(m_1,\ldots,m_K) \geq (0,\ldots,0)$, i.e. $\vec{m}\geq \vec{0}$.
Then the overall generating function for the queue length is $P(\vec{z}) := \E[z_1^{M_1}\cdots z_K^{M_K}] = \sum_{d=1}^D P_d(\vec{z})$.
We also define 
\begin{align*}
 R_d(\vec{z}) &:=  \sum_{\vec{m} \geq \vec{0}}  \frac{p_{\vec{m},d} {\vec{z^m}}}{\sum_j g_j m_j}\cdot \vec{1}_{\{\sum_{j=1}^K m_j > 0\}},\quad \text{hence}\\
 \frac{\p R_d(\vec{z})}{\p z_k} &= z_k^{-1} \sum_{\vec{m} \geq \vec{e}_k}\frac{ m_k}{\sum_j g_j m_j}p_{\vec{m},d}{\vec{z^m}}\cdot \vec{1}_{\{\sum_{j=1}^K m_j > 0\}}.
\end{align*}
By multiplying the flow equation \eqref{eq:floweqn} with $\vec{z}^{\vec{m}}$, summing over all vectors $\vec{m}\geq\vec{0}$ and rearranging terms, we can eventually write it in terms of the PGF $P_d(\vec{z})$ and the partial derivative $\p R_d(\vec{z})/\p z_k$, that is,
\begin{equation}\label{eq:pgfpdeRP}
\sum_{k=1}^K \left[ \lambda_{k,d} (1-z_k)P_d(\vec{z}) + \mu_{k,d} g_k (z_k-1)\frac{\p R_d(\vec{z})}{\p z_k}\right] = \sum_{\ell=1}^D P_\ell(\vec{z})q_{\ell d}.
\end{equation}

It will be convenient to write the equation fully in terms of $\p R_d / \p z_k$, so we note the relation
\begin{equation}\label{eq:PpartialR}
P_d(\vec{z}) = \sum_{k=1}^K g_k z_k \frac{\p R_d(\vec{z})}{\p z_k} + p_{\vec{0},d} 
\end{equation}
where $p_{\vec{0},d} = P((\vec{M},Z)=(\vec{0},d))$ is the probability of an empty queue in state $d$, which is equivalent to the probability of no workload defined in Section \ref{sec:workload}.
We incorporate this into Eqn. \eqref{eq:pgfpdeRP} to obtain
\begin{align} \label{eq:multimodPDE}
\sum_{k=1}^K &\lambda_{k,d} (1-z_k)\left[\sum_{j=1}^K g_j z_j \frac{\p R_d(\vec{z})}{\p z_j} + p_{\vec{0},d}\right] \nonumber \\
&+ \sum_{k=1}^K\mu_{k,d} g_k (z_k-1)\frac{\p R_d(\vec{z})}{\p z_k} = \sum_{\ell=1}^D\left[\sum_{k=1}^K g_k z_k \frac{\p R_\ell(\vec{z})}{\p z_k} + p_{\vec{0},\ell}\right]q_{\ell d}.
 \end{align}
 This equation we will use later when deriving the heavy-traffic limit.
 
 For the M/M/1 queue with modulated arrivals and service times, moments of the queue length and the sojourn times can be found for the FCFS service discipline, in \cite{YechNaor} and \cite{Mahab05}, respectively. In \cite{Rege96} the authors establish a recursive formula to calculate moments of the queue length in a non-modulated DPS system. In a similar fashion, we obtain an expression for the sum of the state-dependent moments weighted with the capacity of the server. We also derive a {\it rate conservation law}, which shows how the average arrival rates per class are proportional to the resources allocated to that same class, and the service they receive. Both results can be found in the following proposition.
 
 \begin{proposition}
  When the queue is stable, the average number of class-$k$ arrivals is proportional to the service resources allocated to class-$k$ customers, i.e.
\[
 \lambda_{k,\infty} = \sum_d \mu_{k,d} \E\left[\frac{g_k M_k}{\sum_j g_j M_j}\cdot \vec{1}_{\{\sum_j M_j>0 \}} \cdot \vec{1}_{\{Z=d\}} \right],
\]
for $k=1,\ldots,K$. Furthermore, the state-dependent expectations of $M_k$ satisfy
\[
 \sum_d c_d \E[M_k \cdot \vec{1}_{\{Z=d\}} ]=
\frac{
\lambda_{k,\infty}}{\mu_k} + \sum_{d,j} g_j\frac{\lambda_{k,d}
\E[M_j \cdot \vec{1}_{\{Z=d\}} ]
+ \lambda_{j,d}
\E[M_k\cdot \vec{1}_{\{Z=d\}}]
}
{\mu_kg_k + \mu_jg_j}.
\]
 \end{proposition}

 \begin{proof}

We sum Eqn. \eqref{eq:pgfpdeRP} over $d$ and take the derivative w.r.t. $z_i$ to obtain
\begin{align*}
 \sum_d\left[-\lambda_{i,d}P_d(\vec{z}) + \sum_j \lambda_{j,d}(1-z_j)\frac{\p P_d(\vec{z})}{\p z_j} + \mu_{i,d}g_i \frac{\p R_d(\vec{z})}{\p z_i} + \sum_j \mu_{j,d} g_j (z_j-1)\frac{\p^2 R_d(\vec{z})}{\p z_i\p z_j} \right] = 0.
  \end{align*}
Letting $\vec{z}\to 1$ yields
  \begin{equation}\label{eq:lambdamusto0}
 \sum_d\left[\mu_{i,d}g_i \frac{\p R_d(\vec{z})}{\p z_i}\bigg|_{\vec{z}\to 1} - \lambda_{i,d}P_d(\vec{z})\bigg|_{\vec{z}\to 1}\right] = 0.
  \end{equation}
Since
\[
  \lim_{\vec{z}\to 1} P_d(\vec{z}) = \pi_d,
\]
the following conservation law results from Eqn. \eqref{eq:lambdamusto0}, for $k=1,\ldots,K$:
\begin{align}\label{eq:ratecons}
 \lambda_{k,\infty} &= \sum_d \mu_{k,d}g_k \frac{\p R_d(\vec{z})}{\p z_k}\bigg|_{\vec{z}\to 1} \\
 &= \sum_d \mu_{k,d} \E\left[\frac{g_k M_k}{\sum_j g_j M_j}\cdot \vec{1}_{\{\sum_j M_j>0 \}} \cdot \vec{1}_{\{Z=d\}} \right]. \nonumber
\end{align}

By taking partial derivatives of Eqn. \eqref{eq:PpartialR}, we obtain after standard calculations the recursive relation
\begin{equation}\label{eq:momentrecursion}
 \frac{\p^j P_d(\vec{z})}{\p z_{i_1}\cdots \p z_{i_j}}\bigg|_{\vec{z}\to 1} = \sum_{k=1}^K g_k \frac{\p^{j+1} R_d(\vec{z})}{\p z_{i_1}\cdots \p z_{i_j}\p z_k} 
 + \sum_{\ell=1}^j g_{i_\ell} \frac{\p^j R_d(\vec{z})}{\p z_{i_1}\cdots \p z_{i_j}}.
\end{equation}
In particular, this yields the explicit form
\[
 \E[M_k \cdot \vec{1}_{\{Z=d\}}] = \frac{\p P_d}{\p z_k}\bigg|_{\vec{z}\to 1} = \sum_j g_j\frac{\p^2 R_d}{\p z_k\p z_j}\bigg|_{\vec{z}\to 1} + g_k \frac{\p R_d}{\p z_k}\bigg|_{\vec{z}\to 1}.
\]
Proceeding from the rate conservation law, Eqn. \eqref{eq:ratecons}, and by using $\mu_{k,d} = \mu_k c_d$, we have
\[
 \sum_d c_d g_k \frac{\p R_d}{\p z_k}\bigg|_{\vec{z}\to 1} = \frac{\lambda_{k,\infty}}{\mu_k}.
\]
By taking two partial derivatives of the balance equation Eqn. \eqref{eq:pgfpdeRP} we can solve for a second mixed derivative of $R_d$, namely
\[
 \frac{\p^2 R_d}{\p z_k\p z_j}\bigg|_{\vec{z}\to 1} 
 = \frac{
 \sum_\ell 
 \E[M_k M_j\cdot \vec{1}_{\{Z=\ell\}}]
 q_{\ell d} + \lambda_{k,d}
 \E[M_j \cdot \vec{1}_{\{Z=d\}}]
 + \lambda_{j,d}
 \E[M_k \cdot \vec{1}_{\{Z=d\}}]
 }
 {\mu_{k,d}g_k + \mu_{j,d}g_j},
\]
thus also yielding a mixed moment.
Summing over the weighted moments of the number of class-$k$ customers while in state $d$, we obtain a linear equation resembling Eqn. (16) in \cite{Rege96} for the non-modulated DPS queue:
\begin{align*}
 \sum_d &c_d 
 \E[M_k \cdot \vec{1}_{\{Z=d\}}]
 = \frac{\lambda_{k,\infty}}{\mu_k} + \sum_{d,j} c_d g_j\frac{\p^2 R_d}{\p z_k\p z_j}\bigg|_{\vec{z}\to 1} \quad \text{(by Eqn. \eqref{eq:momentrecursion})}\\
 &= \frac{\lambda_{k,\infty}}{\mu_k} + \sum_{d,j} g_j\frac{\sum_\ell 
 \E[M_k M_j \cdot \vec{1}_{\{Z=\ell\}}]
 q_{\ell d} + \lambda_{k,d}
 \E[M_j      \cdot \vec{1}_{\{Z=d\}} ]
 + \lambda_{j,d}
 \E[M_k\cdot \vec{1}_{\{Z=d\}}]
 }{\mu_kg_k + \mu_jg_j} \\
 &= \frac{\lambda_{k,\infty}}{\mu_k} + \sum_{d,j} g_j\frac{\lambda_{k,d}
 \E[M_j\cdot \vec{1}_{\{Z=d\}}]
 + \lambda_{j,d}
 \E[M_k \cdot \vec{1}_{\{Z=d\}}]
 }{\mu_kg_k + \mu_jg_j},
\end{align*}
where the last equality comes from $\sum_d q_{\ell d} = 0$.
 \end{proof}

 \section{Preliminary results for the queue length in heavy traffic}\label{sec:HT}
We proceed to show that in heavy traffic, the distribution of the environment and the joint queue length become independent. This result, along with two technical lemmas that we derive in this section, will later help to establish the main result about the limiting queue length under DPS, presented in Section \ref{sec:DPSproof}.
Here we consider the queue length vector $(M_1,\ldots,M_K)$ scaled with 
$1/N$ and evaluate the PGF in $z^{1/N}$. 
The objective is to 
determine the distribution of $\frac{1}{N}(M_1^{(N)},\ldots,M_K^{(N)})\cdot\vec{1}_{\{Z=d\}}$ 
as $N$ goes to infinity.
We will state the existence of the limiting vector, and thus also the limit of the generating function $P^{(N)}_d(\vec{z}^{1/N})$, as an assumption. This assumption will be proven in Section~\ref{sec:commonfactor}. The superscript $N$ denotes dependency on the prelimit parameter $\lambda^{(N)}_d$.
  
We make use of the change of variables $\mathrm{e}^{-s_k}=z_k$, for $s_k > 0$, and denote $\mathrm{e}^{-\vec{s}/N} := (\mathrm{e}^{-s_1/N},\ldots,\mathrm{e}^{-s_K/N})$. 
Assuming that the limit exists, we use the new variables to define the heavy-traffic quantities: 
We let $\hat{p}_{\vec{0},d}:= \lim_{N\to\infty} p_{\vec{0},d}^{(N)}$,
 $\hat{P}_d(\vec{s}) := \lim_{N\to\infty} P^{(N)}_d(\mathrm{e}^{-\vec{s}/N}) = \lim_{N\to\infty} \E[\mathrm{e}^{-\sum_j s_j {M}_j/N}\cdot  \vec{1}_{\{Z=d\}}]$ and $\hat{P}(\vec{s}) := \sum_d \hat{P}_d(\vec{s})=  \lim_{N\to\infty} \E[\mathrm{e}^{-\sum_j s_j {M}_j/N} ]$.
We denote by $(\hat M_1,\ldots, \hat M_K)$ the random vector corresponding to the LST $\hat{P}(\vec{s})$.
Finally, let 
 \[
 \hat{R}_d(\vec{s})
 := \E\left[\frac{1 - \mathrm{e}^{-\sum_j s_j \hat{M}_j}}{\sum_j g_j \hat{M}_j}\cdot \vec{1}_{\{\sum_j \hat{M}_j>0\}}\cdot \vec{1}_{\{Z=d\}}\right],
 \]
where the 1 in the numerator is to ensure that the bracketed expression remains bounded when the queue length quantities $\hat{M}_k$ are all near zero. We can now proceed to the following lemma.

\begin{lemma}\label{lem:limPd}
If $\lim_{N\to\infty} P^{(N)}_d(\mathrm{e}^{-\vec{s}/N})$ exists, then it satisfies
\begin{equation}\label{eq:lemma1}
\hat{P}_d(\vec{s}) = \sum_{k=1}^K g_k \frac{\p \hat{R}_d(\vec{s})}{\p s_k}  +  \hat{p}_{\vec{0},d}.
\end{equation}
\end{lemma}

\begin{proof}
From Eqn. \eqref{eq:PpartialR},
\begin{equation}\label{eq:limPpartialR}
 \lim_{N\to\infty} P^{(N)}_d(\mathrm{e}^{-\vec{s}/N}) 
 = \lim_{N\to\infty}\sum_{k=1}^K g_k z_k \frac{\p R^{(N)}_d(\vec{z})}{\p z_k}\bigg |_{\vec{z}=\mathrm{e}^{-\vec{s}/N}} + \hat{p}_{\vec{0},d}.
\end{equation}
Note that
\begin{align} \nonumber
 \lim_{N\to\infty} &z_k\frac{\p R^{(N)}_d(\vec{z})}{\p z_k} \bigg |_{\vec{z}=\mathrm{e}^{-\vec{s}/N}} \\
 &= \lim_{N\to\infty} \sum_{\vec{m} \geq \vec{e}_k}\frac{ m_k}{\sum_j g_j m_j}p^{(N)}_{\vec{m},d}{\vec{z^m}}
 \cdot\vec{1}_{\{\sum_j m_j > 0\}} \bigg |_{\vec{z}=\mathrm{e}^{-\vec{s}/N}} \nonumber \\
 &= \lim_{N\to\infty}\sum_{\vec{m}\geq \vec{e}_k} \frac{m_k}{\sum_j g_j m_j}p^{(N)}_{\vec{m},d}\mathrm{e}^{-s_1m_1/N}\cdot\cdots\cdot \mathrm{e}^{-s_Km_K/N}\cdot\vec{1}_{\{\sum_j m_j >0\}} \nonumber \\
 &= \lim_{N\to\infty} \E\left[\frac{M_k^{(N)}}{\sum_j g_j M_j^{(N)}}\mathrm{e}^{-\sum_js_jM_j^{(N)}/N}
 \cdot\vec{1}_{\{\sum_j M_j^{(N)}/N >0\}} \cdot \vec{1}_{\{Z=d\}}\right] \nonumber \\
 &= \E\left[\frac{\hat{M}_k}{\sum_j g_j \hat{M}_j}\mathrm{e}^{-\sum_js_j\hat{M}_j}
 \cdot\vec{1}_{\{\sum_j \hat{M}_j >0\}} \cdot \vec{1}_{\{Z=d\}}\right] \nonumber \\ \label{eq:limRhat}
 &= \frac{\p \hat{R}_d(\vec{s})}{\p s_k}.
\end{align}

The second-to-last step follows from the fact that $\frac{M_k^{(N)}}{\sum_j g_j M_j^{(N)}}\cdot \mathrm{e}^{-\sum_j s_jM_j^{(N)}}\cdot\vec{1}_{\{\sum_jM_j^{(N)}>0\}}$ is upper bounded by $1/\min_j(g_j)$. By the continuous mapping theorem (see Billingsley's \cite{Billingsley}), it converges in distribution to $\frac{\hat{M}_k}{\sum_j g_j \hat{M}_j}\cdot \mathrm{e}^{-\sum_js_j\hat{M}_j}\cdot\vec{1}_{\{\sum_j \hat{M}_j >0\}}$. The environment is not affected by the heavy-traffic scaling. 
Eqns. \eqref{eq:limPpartialR} and \eqref{eq:limRhat} now conclude the proof.
\end{proof}
With the help of Lemma \ref{lem:limPd} we obtain:

\begin{proposition}\label{prop:decouple}
 If $\lim_{N\to\infty} P^{(N)}_d(\mathrm{e}^{-\vec{s}/N})$ exists, the joint queue length distribution is independent of the environment in heavy traffic, that is
 \[
  \hat{P}_d(\vec{s}) = \pi_d\hat{P}(\vec{s}).
 \]
\end{proposition}

\begin{proof}
We use the change of variables $z_k=\mathrm{e}^{-s_k}$. Since  
\[
 z_k\frac{\p R_d}{\p z_k}(\vec{z}) \bigg |_{\vec{z}=\mathrm{e}^{-\vec{s}}}
 =  -\frac{\p R_d}{\p s_k}(\mathrm{e}^{-\vec{s}}),
\]
we obtain, by applying the heavy traffic scaling to Eqn. \eqref{eq:multimodPDE},

\begin{align*}
\sum_k \left[ \lambda^{(N)}_{k,d} (1-\mathrm{e}^{-s_k/N})\right. &\left[\sum_{j=1}^K g_j
\frac{- \p R_d^{(N)}}{\p s_j}(\mathrm{e}^{-\vec{s}/N}) + p_{\vec{0},d}^{(N)}\right] 
- \left.\mu_{k,d} g_k (\mathrm{e}^{-s_k/N}-1)
\mathrm{e}^{s_k/N} 
\frac{\p R_d^{(N)}}{\p s_k}(\mathrm{e}^{-\vec{s}/N})\right] \\
 &= \sum_{\ell=1}^D\left[\sum_{k=1}^K g_k \frac{- \p R_\ell^{(N)}}{\p s_k}(\mathrm{e}^{-\vec{s}/N}) + p_{\vec{0},\ell}^{(N)}\right]q_{\ell d}. \\
\end{align*}
With Taylor expansion we obtain 
 
 \begin{align}
\sum_k \left[ \lambda^{(N)}_{k,d} \left(\frac{s_k}{N}  - \frac{s_k^2}{N^2}\right)\right. &\left[\sum_{j=1}^K g_j \frac{- \p R_d^{(N)}}{\p s_j}(\mathrm{e}^{-\vec{s}/N}) + p_{\vec{0},d}^{(N)}\right] 
- \left.\mu_{k,d} g_k \left(\frac{s_k}{N} + \frac{s^2_k}{N^2}\right)
\frac{- \p R_d^{(N)}}{\p s_k}(\mathrm{e}^{-\vec{s}/N})\right] \label{eq:dirLimit} \\
&= \sum_{\ell=1}^D\left[\sum_{k=1}^K g_k 
\frac{- \p R_\ell^{(N)}}{\p s_k}(\mathrm{e}^{-\vec{s}/N}) + p_{\vec{0},\ell}^{(N)}\right]q_{\ell d} + \bigO(N^{-3}).  \nonumber
 \end{align}
 
 Since $\frac{- \p R_d^{(N)}}{\p s_j}$ is bounded (see proof of Lemma \ref{lem:limPd}) and converges to $\frac{\p \hat{R}_d}{\p s_j}$,
we obtain as $N\to\infty$,
  \begin{equation}\label{eq:nuQ}
  \vec{\nu}\cdot [Q]_d = \sum_{\ell=1}^D\left[\sum_{k=1}^K g_k \frac{\p \hat{R}_\ell(\vec{s})}{\p s_k} + \hat{p}_{\vec{0},\ell}\right]q_{\ell d} = 0,\quad \vec{s}\geq \vec{0}, \quad\forall d,
 \end{equation}
where $[Q]_d$ is the $d^{th}$ column of $Q$ and $\vec{\nu}$ is a row vector with $\nu_\ell = \sum_{k=1}^K g_k \frac{\p \hat{R}_\ell(\vec{s})}{\p s_k} + \hat{p}_{\vec{0},\ell}$. This implies that $\vec{\nu}Q = 0$, and since $Q$ is a generator we conclude that 
\[
 \nu_d =  \sum_{k=1}^K g_k \frac{\p \hat{R}_d(\vec{s})}{\p s_k} + \hat{p}_{\vec{0},d} = \pi_d x,
\]
where $x$ does not depend on $d$.
Observe now that by Lemma \ref{lem:limPd}, we have
\begin{align*}
\hat{P}_d(\vec{s}) - \hat{p}_{\vec{0},d} 
 &= \sum_{k=1}^K g_k \frac{\p \hat{R}_d(\vec{s})}{\p s_k} \\
 &= \pi_d x  - \hat{p}_{\vec{0},d}  \\
 &= \E[\vec{1}_{\{Z=d\}}] x - \hat{p}_{\vec{0},d}.
\end{align*}

Since $\hat{P}_d(\vec{s}) = \E\left[\mathrm{e}^{-\sum_js_j\hat{M}_j}\cdot \vec{1}_{\{Z=d\}}\right]$, this implies that $x = \E\left[\mathrm{e}^{-\sum_js_j\hat{M}_j}\right] = \hat{P}(\vec{s})$.
This shows that the environment becomes independent from the joint queue-length process in the heavy-traffic limit.
\end{proof}

The flow equation, Eqn. \eqref{eq:multimodPDE}, simplifies considerably in heavy traffic, as shown in the following lemma.
\begin{lemma}\label{lem:Fkd}
  If $\lim_{N\to\infty} P^{(N)}_d(\mathrm{e}^{-\vec{s}/N})$ exists, then $\hat{R}_d(\vec{s})$ satisfies the following equation:
  \[
   0 = \sum_{k,d} F_{k,d}(\vec{s})\frac{\p \hat{R}_d(\vec{s})}{\p s_k}, \quad \forall \vec{s} \geq \vec{0},
  \]
with $F_{k,d}(\vec{s})$ defined as
\[
 F_{k,d}(\vec{s}):=g_k \left( \sum_j \hat{\lambda}_{j,d}s_j - \mu_{k,d}s_k\right).
\]

\end{lemma}
\begin{proof}
We start by multiplying through Eqn. \eqref{eq:dirLimit} with $N$, followed by summing over $d$. Due to $Q$ being a generator, this eliminates the right hand side with the transition rates $q_{\ell d}$:

\begin{align*}
 \sum_{k,d} &\left[ \lambda^{(N)}_{k,d} \left(s_k - \frac{s_k^2}{N}\right) \left[\sum_{j=1}^K g_j \frac{- \p R_d^{(N)}}{\p s_j}(\mathrm{e}^{-\vec{s}/N}) + p_{\vec{0},d}^{(N)}\right] 
 \right] \\ &- \sum_{k,d}\left[
 \mu_{k,d} g_k \left(s_k + \frac{s^2_k}{N}\right)
\frac{-\p R_d^{(N)}}{\p s_k}(\mathrm{e}^{-\vec{s}/N})\right] + \bigO(N^{-2}) = 0.
\end{align*}

Taking the limit $N\to\infty$ yields
\begin{align}\nonumber
 0 &= \sum_{k,d} \hat{\lambda}_{k,d}s_k \sum_j g_j \frac{\p \hat{R}_d(\vec{s})}{\p s_j} - \sum_{k,d}\mu_{k,d} g_k s_k \frac{\p \hat{R}_d(\vec{s})}{\p s_k}\\ \nonumber
 &= \sum_{k,d} g_k \frac{\p \hat{R}_d(\vec{s})}{\p s_k}\sum_j \hat{\lambda}_{j,d}s_j  - \sum_{k,d}\mu_{k,d} g_k s_k \frac{\p \hat{R}_d(\vec{s})}{\p s_k} \\ \nonumber
 &= \sum_{k,d} g_k \left( \sum_j \hat{\lambda}_{j,d}s_j - \mu_{k,d}s_k\right)\frac{\p \hat{R}_d(\vec{s})}{\p s_k} \\
 &= \sum_{k,d} F_{k,d}(\vec{s})\frac{\p \hat{R}_d(\vec{s})}{\p s_k}. \label{eq:Modlemma2}
\end{align}
\end{proof}

In what follows we focus on
\[
 F_{k,\infty}(\vec{s}):=\sum_d F_{k,d}(\vec{s}) \pi_d = g_k \left( \sum_j \hat{\lambda}_{j,\infty}s_j - \mu_{k,\infty}s_k\right),
\]
and denote its vector counterpart by $\vec{F}_\infty(\vec{s}) = (F_{1,\infty}(\vec{s}),\ldots,F_{K,\infty}(\vec{s}))$.

\section{Queue length distribution in heavy traffic}\label{sec:DPSproof}
We now state and consequently prove our main result about the queue length distribution.
\begin{theorem}\label{thm:main}
 When scaled by $1/N=(1-\rho^{(N)}_\infty)$, the queue-length vector converges in distribution as $(\rho_{1,\infty}^{(N)},\ldots,\rho_{K,\infty}^{(N)}) \to (\hat{\rho}_{1,\infty},\ldots,\hat{\rho}_{K,\infty})$ i.e.,  $\rho^{(N)}_\infty \to 1$, namely
\begin{equation}
\label{eq:ssc}
\frac{1}{N}(M_1^{(N)}, \ldots, M_K^{(N)})\cdot \vec{1}_{\{Z=d\}} \stackrel{\rm{d}}{\to}      \left(\hat{M}_1,\ldots,\hat{M}_K  \right)\cdot \vec{1}_{\{Z=d\}} \stackrel{\rm d}{=} \pi_d\left(\frac{\hat{\rho}_{1,\infty}}{g_1},\ldots,\frac{\hat{\rho}_{K,\infty}}{g_K}\right) \cdot X
\end{equation}
where $\stackrel{\rm d}{\to}$ denotes convergence in distribution  and $X$ is exponentially distributed with mean
\begin{equation}\label{eq:meanX}
 \E X = \frac{\sum_k \hat{\rho}_{k,\infty} / \mu_k - \sum_d c_d\pi_d a_d(1-\hat{\rho}_d )}{c_\infty \sum_k \hat{\rho}_{k,\infty}/(g_k\mu_k)},
\end{equation}
with $\hat{\rho}_d := c_d^{-1} \sum_k \hat{\lambda}_{k,d} /\mu_k$ and $\vec{a} = (a_1,\ldots,a_D)^T$  
being a solution of
\[
 [Q\cdot\vec{a}]_d =    c_d(1 - \hat{\rho}_d).
\]
\end{theorem}
We will prove this theorem in the two following subsections, first showing in Section \ref{sec:SSC} the state-space collapse observed in Eqn. \eqref{eq:ssc} and then in Section \ref{sec:commonfactor} we will show that $X$ is exponentially distributed with the mean given by Eqn. \eqref{eq:meanX}.

\subsection{State-space collapse}\label{sec:SSC}
The first part of the proof of Theorem \ref{thm:main} is the state-space collapse. In this section, we assume $\lim_{N\to\infty} P_d^{(N)}(\mathrm{e}^{-\vec{s}/N})$ exists.

Observe that due to Proposition \ref{prop:decouple},

\begin{align}
\hat{R}_d(\vec{s}) &= \E\left[\frac{1 - \mathrm{e}^{-\sum_k s_k \hat{M}_k}}{\sum_k g_k \hat{M}_k}\cdot \vec{1}_{\{\sum_k \hat{M}_k>0\}}\cdot\vec{1}_{\{Z=d\}}\right] \nonumber \\
&= \E\left[\frac{1 - \mathrm{e}^{-\sum_k s_k \hat{M}_k}}{\sum_k g_k \hat{M}_k}\cdot \vec{1}_{\{\sum_k \hat{M}_k>0\}}\right]\cdot \pi_d \nonumber \\
&=: \hat{R}(\vec{s})\pi_d, \label{eq:Rinfty}
\end{align}

where the last equation defines $\hat{R}(\vec{s})$, which is independent of $d$. We now derive some properties of $\hat{R}(\vec{s})$.
\begin{lemma}\label{lem:constHyperplane}
$\hat{R}(\vec{s})$ is constant on a $(K-1)$-dimensional hyperplane $\mathcal{H}_c$, where 
\[
 \mathcal{H}_c := \left\{\vec{s} \geq \vec{0} : \sum_k \frac{\hat{\rho}_{k,\infty}}{g_k} s_k = c \right\},\quad c>0.
\]
\end{lemma}
\begin{proof}
 We follow closely the steps of the proof of Lemma 3 in \cite{Verloop11}. The proof has 3 steps: (i) Show that $F_{k,\infty}(\vec{s})$ is parallel to the hyperplane. Hence, any flow corresponding to $F_{k,\infty}$ that starts in the plane, stays in the plane. (ii) Show that $\hat{R}(\vec{s})$ is constant along each flow in the hyperplane and
 (iii) show that each flow in the hyperplane converges to a unique point. This implies that $\hat{R}(\vec{s})$ is constant on the hyperplane.\\ 
 
{\bf (i) $\vec{F}_\infty(\vec{s})$ is parallel to $\mathcal{H}_c$}\\
Observe that with $1 = \sum_k \hat{\rho}_{k,\infty}$ and $\hat{\rho}_{k,\infty} = \hat{\lambda}_{k,\infty}/\mu_{k,\infty}$,
\begin{align*}
 \sum_k\frac{\hat{\rho}_{k,\infty}}{g_k} F_{k,\infty}(\vec{s}) &= \sum_k \hat{\rho}_{k,\infty} \left( \sum_j \hat{\lambda}_{j,\infty}s_j - \mu_{k,\infty}s_k\right) \\
 &=  \sum_j \hat{\lambda}_{j,\infty}s_j - \sum_k \hat{\rho}_{k,\infty}\mu_{k,\infty} s_k \\
 &= \sum_k \hat{\lambda}_{k,\infty}s_k - \sum_k \hat{\lambda}_{k,\infty}s_k \\
 &= 0.
\end{align*}

This indicates that the $K$-dimensional vector $F_\infty(\vec{s})$ is parallel to the hyperplane.\\

{\bf (ii) $\hat{R}(\vec{s})$ is constant along flows in $\mathcal{H}_c$}\\
For each state $\vec{s}\geq\vec{0}$, there exists a unique flow $\vec{f}(u) = (f_1(u),\ldots,f_K(u))^T$ parametrized by $u\geq 0$, such that
\begin{equation}\label{eq:flow}
  \vec{f}(0) = \vec{s} \quad \text{and} \quad \frac{{\rm d}f_k(u)}{{\rm d}u} = F_{k,\infty}(\vec{f}(u)).
\end{equation}
Due to (i), any flow that starts in $\mathcal{H}_c$, stays in $\mathcal{H}_c$. Now,
\begin{align*}
 \frac{{\rm d} \hat{R}(\vec{f}(u))}{{\rm d} u} &= \sum_{k=1}^K \frac{{\rm d}f_k(u)}{{\rm d}u}\cdot \frac{\p \hat{R}(\vec{s})}{\p s_k}\bigg |_{\vec{s}=\vec{f}(u)} \\
 &= \sum_{k=1}^K F_{k,\infty}(\vec{f}(u))\cdot \frac{\p \hat{R}(\vec{s})}{\p s_k}\bigg |_{\vec{s}=\vec{f}(u)} \\
 &= \sum_{k=1}^K\sum_{d=1}^D F_{k,d}(\vec{f}(u))\pi_d\cdot \frac{\p \hat{R}(\vec{s})}{\p s_k}\bigg |_{\vec{s}=\vec{f}(u)} \\
 &= \sum_{k=1}^K\sum_{d=1}^D F_{k,d}(\vec{f}(u))\cdot \frac{\p \hat{R}_d(\vec{s})}{\p s_k}\bigg |_{\vec{s}=\vec{f}(u)} \\
  &= 0, \qquad \text{by Eqn. \eqref{eq:Modlemma2}},
\end{align*}
implying that $\hat{R}(\vec{f}(u))$ is constant along each flow $\vec{f}(u)$ which lies in $\mathcal{H}_c$.\\

{\bf (iii) Each flow in $\mathcal{H}_c$ converges to a unique point}\\
Here we first write the flow specifications in a vector-matrix form, then show that one eigenvalue of that matrix is zero with eigenvector $\vec{s}^*\in \mathcal{H}_1$, and the other eigenvalues are negative, and thus we can write $\vec{f}(u) = c\cdot \vec{s}^* + \vec{g}(u)$ where $\lim_{u\to\infty}\vec{g}(u) = 0$.\\
Eqn. \eqref{eq:flow} can be written in matrix-vector form as
\[
 \vec{f}'(u) = A \vec{f}(u),
\]
with
\[
 A = \left( \begin{array}{cccc}
       g_1(\hat{\lambda}_{1,\infty} - \mu_{1,\infty}) & g_1\hat{\lambda}_{2,\infty} & \cdots & g_1\hat{\lambda}_{K,\infty} \\
       g_2\hat{\lambda}_{1,\infty} & g_2(\hat{\lambda}_{2,\infty} - \mu_{2,\infty}) & \cdots & g_2\hat{\lambda}_{K,\infty} \\
       \vdots & & \ddots & \vdots \\
       g_K\hat{\lambda}_{1,\infty} & \cdots & & g_K(\hat{\lambda}_{K,\infty} - \mu_{K,\infty})
            \end{array}
\right).
\]
Let $D$ be the diagonal matrix with $d_i=\hat{\rho}_{i,\infty}/g_i$ on the diagonal. Then with 
\[
S:= DAD^{-1} =  \left( \begin{array}{cccc}
       g_1(\hat{\lambda}_{1,\infty} - \mu_{1,\infty}) & g_2\frac{\hat{\rho}_{1,\infty}}{\hat{\rho}_{2,\infty}}\hat{\lambda}_{2,\infty} & \cdots & g_K\frac{\hat{\rho}_{1,\infty}}{\hat{\rho}_{K,\infty}}\hat{\lambda}_{K,\infty} \\
       g_1\frac{\hat{\rho}_{2,\infty}}{\hat{\rho}_{1,\infty}}\hat{\lambda}_{1,\infty} & g_2(\hat{\lambda}_{2,\infty} - \mu_{2,\infty}) & \cdots & g_K\frac{\hat{\rho}_{2,\infty}}{\hat{\rho}_{K,\infty}}\hat{\lambda}_{K,\infty} \\
       \vdots & & \ddots & \vdots \\
       g_1\frac{\hat{\rho}_{K,\infty}}{\hat{\rho}_{1,\infty}}\hat{\lambda}_{1,\infty} & \cdots & & g_K(\hat{\lambda}_{K,\infty} - \mu_{K,\infty})
            \end{array}
\right),
\]
$S^T$ is a generator corresponding to a finite-state Markov chain. 

From the proof of Lemma 4 in \cite{Verloop11}, it is easily seen that the Markov chain corresponding to $S^T$ is irreducible (since we assume that all $\lambda_{k,d}>0$ for all $k$ and at least one $d$). Retracing the arguments stated there, for completeness, it follows that this Markov chain has a unique equilibrium distribution (column) vector, $\vec{\eta}$, such that $\vec{\eta}^TS^T=0$. In particular, 0 is an eigenvalue with multiplicity one and all other eigenvalues have a strictly negative real part, see \cite{ASMU}. Since the eigenvalues of $S^T$ and $A$ are the same, 0 is also an eigenvalue of $A$ with corresponding right eigenvector $\vec{s}^* = D^{-1}\vec{\eta}$, $\vec{s}^*\geq\vec{0}, \vec{s}^*\in \mathcal{H}_1$. The solution of the linear system $\vec{f}'(u) = A \vec{f}(u), \vec{f}(0)\in \mathcal{H}_c$ can now be written as the sum of the homogeneous and the particular solution, i.e. $\vec{f}(u) = c\cdot \vec{s}^* + \vec{g}(u)$, where $\lim_{u\to\infty} \vec{g}(u) = \vec{0}$. This implies that all the flows in $\mathcal{H}_c$ converge to one common point $c\cdot \vec{s}^*$.\\
Combining (i), (ii) and (iii), we conclude that the function $\hat{R}(\vec{s})$ is constant on $\mathcal{H}_c$.
\end{proof}

As a consequence of Lemma \ref{lem:constHyperplane}, the function $\hat{R}(\vec{s})$ depends on $\vec{s}$ only through the sum $\sum_{k=1}^K (\hat{\rho}_{k,\infty}/g_k)s_k$. Therefore, there exists a function $\hat{R}^*: \R\to\R$ such that $\hat{R}(\vec{s}) = \hat{R}^*(\sum_{k=1}^K (\hat{\rho}_{k,\infty}/g_k)s_k)$. Then 
\[
\frac{\p}{\p s_k}\hat{R}(\vec{s}) = \frac{\hat{\rho}_{k,\infty}}{g_k} \frac{{\rm d} \hat{R}^*(v)}{{\rm d} v}\bigg|_{v=\sum_{k=1}^K (\hat{\rho}_{k,\infty}/g_k)s_k},  
\]
so we obtain
\begin{align}\nonumber
 \E[\mathrm{e}^{-\sum_{k=1}^K s_k \hat{M}_k}] \nonumber
 &= \lim_{N\to\infty}\sum_{d=1}^D P^{(N)}_d(\mathrm{e}^{-\vec{s}/N}) \\ \nonumber
 &= \sum_{d=1}^D\sum_{k=1}^K g_k \frac{\p \hat{R}_d(\vec{s})}{\p s_k} \\ \nonumber
 &= \sum_{d=1}^D\sum_{k=1}^K g_k \frac{\p \hat{R}(\vec{s})}{\p s_k}\pi_d \\ \nonumber
 &= \sum_{k=1}^K \hat{\rho}_{k,\infty} \frac{{\rm d} \hat{R}^*(v)}{{\rm d} v}\bigg|_{v=\sum_{k=1}^K (\hat{\rho}_{k,\infty}/g_k)s_k} \\ \label{eq:dRstarv}
 &= \frac{{\rm d} \hat{R}^*(v)}{{\rm d} v}\bigg|_{v=\sum_{k=1}^K (\hat{\rho}_{k,\infty}/g_k)s_k},
\end{align}
which only depends on $v=\sum_{k=1}^K (\hat{\rho}_{k,\infty}/g_k)s_k$. Since we also have
\begin{align*}
 \E[\mathrm{e}^{-\sum_{k=1}^K s_k \hat{M}_k}] &= \E\left[\mathrm{e}^{-\frac{g_1}{\hat{\rho}_{1,\infty}} v \hat{M}_1} \cdot 
 \mathrm{e}^{-\frac{\hat{\rho}_{2,\infty}}{g_2}s_{2} \left(\frac{g_2}{\hat{\rho}_{2,\infty}} \hat{M}_2 - \frac{g_1}{\hat{\rho}_{1,\infty}} \hat{M}_1\right)} \right. \\
 &\cdot \cdots \cdot 
 \left. \mathrm{e}^{-\frac{\hat{\rho}_{K,\infty}}{g_K}s_{K} \left(\frac{g_K}{\hat{\rho}_{K,\infty}} \hat{M}_K - \frac{g_1}{\hat{\rho}_{1,\infty}} \hat{M}_1\right)}\right],
 \end{align*}
 this together with Eqn. \eqref{eq:dRstarv} implies that $\left(\frac{g_j}{\hat{\rho}_{j,\infty}} \hat{M}_j - \frac{g_1}{\hat{\rho}_{1,\infty}} \hat{M}_1\right)=0$, for all $j=1,\ldots, K$. Thus
 $(g_k/\hat{\rho}_{k,\infty})\hat{M}_k \stackrel{\rm d}{=}(g_j/\hat{\rho}_{j,\infty})\hat{M}_j$, for all $k,j$, almost surely. Combining this finding with that of Eqn. \eqref{eq:nuQ}, we obtain Eqn. \eqref{eq:ssc} 
with $X$ distributed as $(g_1/\hat{\rho}_{1,\infty})\hat{M}_1$.

\begin{remark}[Continuously modulated service requirements]
\label{rem:split}
 In Section \ref{sec:workload} we saw that the critical load is indeed reached when $\rho_\infty \to 1$, since then $p_{0,d}\to 0$.
This indicates that $(1-\rho_\infty)$ is the right heavy-traffic scaling when $\mu_{k,d} = \mu_k c_d$. This is less clear for a general $\mu_{k,d}$, that is, for continuously modulated service requirements, where the environment can influence the departure rate of customers present in the system. 
For that setting, the workload process is no longer independent of the employed scheduling discipline, since the decision on which class to serve impacts the rate at which customers leave.
We are not aware of any results on workload and waiting time distributions where the service distribution is a general function of both class and environment.

The majority of the preceding queue length results in this paper can however be proven without the restriction of the product form, i.e., for continuously modulated service requirements.
  The traffic intensity for this variant is defined as for the multi-class model above, only this time one cannot split the average class-$k$ service rate into $\mu_{k,\infty} = \mu_k c_\infty$. 
  The traffic intensity per class $k$, 
  $\rho_{k,\infty} = \lambda_{k,\infty}/\mu_{k,\infty}$, 
  is in line with the 
  Markov-modulated single-server queues.
Assuming there exists a scaling $f(N)$ such that  
$f(N) (M_1,\ldots, M_K)\vec{1}_{\{Z=d\}}$ converges in distribution, it can be shown
    that the empty probabilities $p_{\vec{0},d}$ vanish in heavy traffic as $N\to\infty$, for a general $\mu_{k,d}$. This property then follows from Proposition~\ref{prop:decouple} without relying on the workload results from Section \ref{sec:workload} and the product form assumed there.
Furthermore, under this assumption, 
 all results in Section~\ref{sec:SSC} hold, implying that a state-space collapse will appear. In other words, we can prove the first half of Theorem~\ref{thm:main}. However, we do not know what the distribution of the common factor $X$ will be. 
\end{remark}

\subsection{Distribution of the common factor}\label{sec:commonfactor}
In order to prove that the limiting queue length distribution exists and to find the common factor of the queue length distribution in heavy traffic, the random variable $X$, we make use of the results on the workload of the total system. 
From \cite{Verloop11} and Eqn. \eqref{eq:ssc} we have that
\begin{equation}\label{eq:WX}
 \hat{W} \stackrel{\rm d}{=} \sum_{k=1}^K \frac{\hat{M_k}}{\mu_k} = X\cdot \sum_{k=1}^K \frac{\hat{\rho}_{k,\infty}}{g_k\mu_k}.
\end{equation}
In order to apply the workload result of Section~\ref{sec:workload}, we first derive the service requirement of an arbitrary customer while being in state $d$. If $H_k(\cdot)$ is the distribution function of a class-$k$ customer's service requirement, then the probability of a class-$k$ customer arriving and requiring service not exceeding $x$ is $\alpha_{k,d}H_k(x)$. Summing over $k$ now yields the desired distribution,
\[
 H_d(x) := \sum_{k=1}^K \alpha_{k,d} H_k(x).
\]
The overall service requirement distribution thus depends on the state of the environment at its arrival. With exponential service requirements, the corresponding LST is given by
\begin{equation}\label{eq:hd}
h_d(s) = \sum_{k=1}^K \frac{\alpha_{k,d} \mu_k}{\mu_k + s}, \quad s\geq 0,
\end{equation}
and the first and second moment are given by
\begin{equation}\label{eq:momentshd}
 h_{d1} = \sum_k \frac{\alpha_{k,d}}{\mu_k} 
 , \quad h_{d2} = 2\sum_k \frac{\alpha_{k,d}}{\mu_k^2}.
\end{equation}
We can now apply the result of Theorem \ref{thm:wht} with moments as given in Eqn. \eqref{eq:momentshd}. 
Hence, we have that $\hat{W}$ is exponentially distributed with mean
\[
\E \hat{W} = c_\infty^{-1}\left(\sum_k \hat{\rho}_{k,\infty}/\mu_k - \sum_d c_d\pi_d a_d(1-\hat{\rho}_d  )\right),
\]
where $\vec{a}$ is a solution of $[Q\cdot \vec{a}]_d 
 =   c_d - \hat \lambda_d \sum_k \frac{\alpha_{k,d}}{\mu_k} = c_d (1-\hat \rho_d)$. 
Along with Eqn. \eqref{eq:WX} this yields the mean of the exponential random variable $X$:
\begin{align} \nonumber
 \E X &= \frac{\E \hat{W}}{\sum_k \hat{\rho}_{k,\infty}/(g_k\mu_k)}\\
 &= \frac{\sum_k \hat{\rho}_{k,\infty} / \mu_k - \sum_d c_d \pi_d a_d(1-\hat{\rho}_d  )}{c_\infty \sum_k \hat{\rho}_{k,\infty}/(g_k\mu_k)}. \label{eq:EX}
 \end{align}
The first term of the numerator is in accordance with the results of \cite{Rege96} and \cite{Verloop11}, the second term is a result of the random environment.

The results in Sections \ref{sec:HT} and \ref{sec:SSC} are based on the assumption that  
$\lim_{N\to\infty} \frac{1}{N}\vec{M}\cdot\vec{1}_{\{Z=d\}}$
exists. Since the scaled workload is tight, see Section \ref{sec:workload}, so is the scaled queue length. Then, by Prohorov's theorem (\cite{Billingsley}) there exists a subsequence of $N$ such that $\frac{1}{N}M_k$ converges in distribution, and hence for this subsequence $\lim_{N\to\infty} P^{(N)}_d(\mathrm{e}^{-\vec{s}/N})$ exists. Since each converging subsequence yields the same limit, the limit itself exists (see corollary page 59 in~\cite{Billingsley}),
i.e. 
$\frac{1}{N}(\vec{M}, Z=d) \stackrel{\rm d}{\to} \vec{\hat{M}}\cdot\vec{1}_{\{Z=d\}}$,
as $N\to\infty$, with the limiting vector as in Eqn. \eqref{eq:ssc}.

This concludes the proof of Theorem \ref{thm:main}.

\section{Conclusion and future work}\label{sec:conclusion}
We first studied the workload for a queue with modulated arrivals, service requirements and service capacity, and derived that the scaled workload converges to an exponentially distributed random variable in heavy traffic. The workload results obtained are valid for any service distribution and for any service discipline which does not depend on the environment.
We then focussed on the special setting of a multi-class queue under the DPS policy 
and showed that the joint queue length distribution for such a system 
undergoes a state-space collapse in heavy traffic. Under the scaling of $(1-\rho_\infty)$, the vector-valued limiting distribution is independent of the modulating environment and converges in distribution to a one-dimensional random variable times a deterministic vector. In this derivation, the distribution of the scaled workload is a key quantity. With this we extend known results about the DPS queue to a Markov-modulated setting.

Clearly an interesting question for future consideration is whether the
state-space collapse for the DPS policy carries over to continuously modulated service requirements, as discussed in Remark \ref{rem:split}.
Another open question concerns the 
characterization of the moments of the queue lengths for the modulated DPS queue,
outside of heavy traffic. Last but not least, modulating the weights of the DPS would open the possibility of dynamical scheduling based on the environment.
The latter would be a study on its own, as already the stability
conditions will no longer be independent of the weights of the DPS
policy.

\section*{Appendix: Proof of Theorem \ref{thm:wht}}

The proof Theorem \ref{thm:wht} is based on Theorem 4 in \cite{Dimi11}, which can be adapted to our model as follows:

We start with notation and some preliminaries.
Let $\Lambda = \text{diag}(\lambda_1,\ldots,\lambda_D)$, $\bar{H}(s) = \text{diag}(1 - h_1(s),\ldots,1 - h_D(s))$, $C= \text{diag}(c_1,\ldots,c_D)$ and $\vec{p}_0 = (p_{0,1},\ldots,p_{0,D})$. Furthermore $\bar{H}_1$ and $\bar{H}_2$ are the diagonal matrices corresponding to the moments $h_{d1}$ and $h_{d2}$, respectively, for $d=1,\ldots,D$. 
Recall Eqn. \eqref{eq:Qa}, $[Q\cdot \vec{a}]_d = c_d - \lambda_d h_{d1} -c_\infty(1-\rho_\infty)$. We will now construct a partial inverse of $Q$ to make it easier to find a vector $\vec{a}$ which solves this equation. Let $Q_1$ and $R$ be matrices such that 
\[ 
Q_1=
\begin{pmatrix}
q_{22} & q_{23} & \cdots & q_{2D} \\
\vdots & & &\\
q_{D2} & q_{D3} & \cdots & q_{DD}
\end{pmatrix}, \quad
R=
\begin{pmatrix}
0 & 0 & \cdots & 0 \\
0 &&&\\
\vdots & Q_1^{-1} &\\
0 & &&
\end{pmatrix}.
\]
Then $\det Q_1 \neq 0$ and due to $Q$ being a generator (for more details see \cite{Dimi11}), we have
\[
 QR = \begin{pmatrix}
       0 & \frac{-\pi_2}{\pi_1} & \frac{-\pi_3}{\pi_1} & \cdots & \frac{-\pi_D}{\pi_1} \\
       0 & 1 & 0 & \cdots & 0 \\
       0 & 0 & 1 & 0 & \cdots \\
       \vdots & \vdots & \vdots & \ddots & \ddots \\
       0 & 0 & \cdots & \cdots & 1
      \end{pmatrix}.
\]
It follows that for any vector $\vec{x}$, it holds that 
\begin{equation}\label{eq:yQR}
\vec{x}QR = \vec{x} - \frac{x_1}{\pi_1} \vec{\pi}.
\end{equation}
Then it can be verified with straightforward calculations that
\begin{align} \nonumber
\vec{a} &= (a_1,\ldots,a_D) = R(c_1 - \lambda_1 h_{11} - c_\infty(1-\rho_\infty),\ldots,c_D - \lambda_D h_{D1} - c_\infty(1-\rho_\infty))^T \\ \nonumber
&= R[C - \Lambda \bar{H}_1]\vec{e} - c_\infty(1-\rho_\infty)R\vec{e} \\  
&= R[C - \Lambda \bar{H}_1]\vec{e} - \vec{r} \label{eq:avec}
\end{align}
is a possible solution vector, with $\vec{r} := c_\infty(1-\rho_\infty)R\vec{e}$.
 
 Define the vector $\vec{\varphi} = (\varphi_1(s),\ldots,\varphi_D(s))$ and write 
Eqn. \eqref{eq:phiLST} 
 in matrix-vector terms,
\begin{equation}\label{eq:phiQmatvec}
 \vec{\varphi}(s) Q = \vec{\varphi}(s)[\Lambda \bar{H}(s) - sC] + s \vec{p}_0 C.
 \end{equation}

Observe that, according to 
Eqn. \eqref{eq:pd0cdlimit} 
\[
\vec{p}_0 C\vec{e} = c_\infty(1-\rho_\infty).
\]
Now multiply from the right both sides of the new vector-matrix equation, Eqn. \eqref{eq:phiQmatvec}, with a $D$-dimensional vector of 1's, $\vec{e}$, to obtain
\begin{equation}\label{eq:DimiEqn8}
\vec{\varphi}(s)[\Lambda \bar{H}(s) - sC]\vec{e} + s c_\infty(1-\rho_\infty) = 0.
\end{equation}

Multiply from the right both sides of Eqn. \eqref{eq:phiQmatvec} with the matrix $R$ to get
\[
\vec{\varphi}(s)QR = \vec{\varphi}(s)[\Lambda \bar{H}(s) - sC]R + s \vec{p}_0 C R.
\]
Rewrite this equation by using the property of Eqn. \eqref{eq:yQR} to obtain
\begin{equation}\label{eq:DimiEqn9}
\vec{\varphi}(s) = \frac{\varphi_1(s)}{\pi_1} \vec{\pi} + \vec{\varphi}(s)[\Lambda \bar{H}(s) - sC]R + s \vec{p}_0 C R.
\end{equation}
Iterate Eqn. \eqref{eq:DimiEqn9} with itself by inserting $\vec{\varphi}(s)$ into the right hand side of the equation to obtain, after some algebraic transformations,
\begin{equation}\label{eq:DimiEqn10}
\vec{\varphi}(s) =  \frac{\varphi_1(s)}{\pi_1} \vec{\pi}[I + G(s)R] + \vec{y}(s)
\end{equation}
with $G(s) := \Lambda \bar{H}(s) - sC$ and 
\begin{equation}\label{eq:ys}
\vec{y}(s) := \vec{\varphi}(s)[G(s)R]^2 + s\vec{p}_0 CR[G(s)R + I].
\end{equation}

Substitute Eqn. \eqref{eq:DimiEqn10} into Eqn. \eqref{eq:DimiEqn8} to obtain an expression for $\varphi_1(s)$,
\begin{align} \nonumber
 0 &= \left[\frac{\varphi_1(s)}{\pi_1} \vec{\pi}[I + G(s)R] + \vec{y}(s)\right] \cdot G(s)\vec{e} + s c_\infty(1-\rho_\infty) \\ \nonumber
 &= \frac{\varphi_1(s)}{\pi_1}\left[\vec{\pi} G(s)\vec{e} + \vec{\pi}G(s)RG(s)\vec{e} \right] + \vec{y}(s)G(s)\vec{e} + s c_\infty(1-\rho_\infty) \\ \label{eq:AB1B2}
 &= \frac{\varphi_1(s)}{\pi_1}\left[B_2(s) + B_3(s)\right] + B_1(s)
\end{align}
with $B_1(s) = \vec{y}(s)G(s)\vec{e} + s c_\infty(1-\rho_\infty)$, $B_2(s) = \vec{\pi}G(s)\vec{e}$ and $B_3(s) = \vec{\pi}G(s)RG(s)\vec{e}$.\\

The next step is to insert the scaling $s \mapsto s/N$ for each term. Recall that using the heavy traffic parametrization introduced in Section \ref{sec:model}, we have $(1-\rho_\infty)=1/N$. 
Now observe that, as $N\to\infty$,
\[
\frac{\bar{H}(s/N)}{s/N} \to \bar{H}_1, \quad \frac{\bar{H}_1s/N - \bar{H}\left(s/N\right)}{(s/N)^2} \to \frac{\bar{H}_2}{2}.
\]
Therefore the limit
\[
 \frac{G(s/N)}{s/N} \to \Lambda \bar{H}_1 - C,
\]
is a constant 
and since $|\vec{\varphi}(s/N)| \leq 1$ and $\vec{p}^{(N)}_0 \to 0$ (see Eqn. \eqref{eq:pd0cdlimit}), we have
\begin{align*}
 \frac{\vec{y}(s/N)}{s/N} &= \vec{\varphi}(s/N)\frac{[G(s/N)R]^2}{s/N} + \vec{p}^{(N)}_0 CR[G(s/N)R + I] \\
 &= \vec{\varphi}(s/N)\left(\frac{s}{N}\right)\left[\frac{G(s/N)}{s/N}R\right]^2 + \vec{p}^{(N)}_0 CR[G(s/N)R + I] \\
 &\to 0,
\end{align*}
as $N\to\infty$.
Combining the above we obtain
\begin{align*}
B_1(s/N) &= \vec{y}(s/N)G(s/N)\vec{e} + sc_\infty(1-\rho_\infty)/N\\
&= sc_\infty(1 - \rho_\infty)/N + o(N^{-2}) = sc_\infty /N^2 + o(N^{-2}).
\end{align*}

Then
\begin{align*}
B_2(s/N) &= \vec{\pi} \left[\Lambda \bar{H}\left(s/N\right) - sC/N \right]\vec{e} \\
&= \vec{\pi} \Lambda \left[\bar{H}\left(s/N\right) - \bar{H}_1s/N\right]\vec{e}
 + \vec{\pi}\left[\Lambda \bar{H}_1s/N - sC/N \right]\vec{e} \\
 &= \vec{\pi} \Lambda \frac{\bar{H}\left(s/N\right) - \bar{H}_1s/N}{(s/N)^2}(s/N)^2\vec{e} - sc_\infty(1 - \rho_\infty)/N\\
 &= -\vec{\pi} \Lambda \frac{\bar{H}_2}{2}(s/N)^2\vec{e} + o(N^{-2}) - sc_\infty /N^2\\
 &= -(s/N)^2\sum_{d=1}^D \pi_d\lambda_d h_{d2}/2 - sc_\infty /N^2  + o(N^{-2})
\end{align*}
since $\vec{\pi} \Lambda \bar{H}_1 \vec{e} = \sum_{d=1}^D \pi_d\lambda_d h_{d1} = \rho_\infty c_\infty$. Furthermore,
\begin{align*}
B_3(s/N) &= \vec{\pi} \left[\Lambda \bar{H}\left(s/N\right) - C s/N \right] R \left[\Lambda \bar{H}\left(s/N\right) - Cs/N \right]\vec{e} \\
&= \vec{\pi} (s/N)^2\left[\Lambda \frac{\bar{H}\left(s/N\right)}{(s/N)} - C\right] R \cdot \left[\Lambda \frac{\bar{H}\left(s/N\right)}{(s/N)} - C \right]\vec{e} \\
&= \vec{\pi} (s/N)^2 \left[\Lambda\bar{H}_1 - C \right] R \left[\Lambda \bar{H}_1 - C\right] \vec{e} + o(N^{-2}) \\
&= -\vec{\pi} (s/N)^2 \left[\Lambda\bar{H}_1 - C \right] (\vec{a} + \vec{r}) + o(N^{-2}) \\
&=  -\sum_d \pi_d\left[(a_d+r_d) (\lambda_d h_{d1} - c_d)\right](s/N)^2 + o(N^{-2})
\end{align*}
due to $ R \left[C - \Lambda \bar{H}_1\right] \vec{e} = \vec{a} + \vec{r}$, see Eqn. \eqref{eq:avec}. Under the heavy-traffic scaling, \[\vec{r} = c_\infty(1-\rho_\infty)R\vec{e} = c_\infty N^{-1}R\vec{e}\] is an $o(1)$ term.
Observe that 
\begin{align*}
 -&(B_2(s/N) + B_3(s/N)) \\
 &= (s/N)^2 \sum_{d=1}^D \pi_d[\lambda_d h_{d2}/2 + (a_d + o(1))(\lambda_dh_{d1} - c_d)] + sc_\infty /N^2 + o(N^{-2}).
\end{align*}

Rearranging Eqn. \eqref{eq:AB1B2} yields 
\begin{align*}
\varphi_1(s/N) &= \pi_1\frac{B_1(s/N)}{-(B_2(s/N) + B_3(s/N))} \\
 &= \pi_1 \frac{c_\infty s/N^2 + o(N^{-2})}{(s/N)^2 \sum_{d=1}^D \pi_d[\lambda_d h_{d2}/2 + a_d(\lambda_dh_{d1} - c_d)] + c_\infty s/N^2 + o(N^{-2}) } \\
  &= \pi_1 \frac{1 + o(1)}{1 + c_\infty^{-1} \sum_d \pi_d \left[\lambda_d h_{d2}/2 + a_d (\lambda_d h_{d1} - c_d)\right]s + o(1)}.
\end{align*}

Let $M$ be the desired mean stated in Theorem 3.3, that is
\[
M 
 := c_\infty^{-1}\sum_d \pi_d\left[\hat{\lambda}_d h_{d2}/2 + a_d (\hat{\lambda}_d h_{d1} - c_d)\right].
 \]
Then, 
taking the heavy-traffic limit,
\[
\lim_{N\to\infty} \vec{\varphi}(s/N) = \lim_{N\to\infty} \frac{\varphi_1(s/N)}{\pi_1}\vec{\pi} = \frac{\vec{\pi}}{1 + Ms}
\]
i.e. the LST $\vec{\varphi}(s)$ converges in distribution to the LST of an exponentially distributed random variable with mean $M$.

\end{document}